\documentclass[reqno]{amsart}
\usepackage{amsmath}
\usepackage[dvips]{graphicx}
\usepackage{amsfonts}
\usepackage{amssymb}
\usepackage{latexsym}
\graphicspath{{Img/}}
\newtheorem{theorem}{Theorem}
\theoremstyle{plain}

\newtheorem{corollary}{Corollary}

\newtheorem{proposition}{Proposition}

\numberwithin{equation}{section}

\begin{document}

\title[Platonic geodesic trees]{Platonic geodesic trees in the $K-$Space}
\author{Anastasios N. Zachos}
\address{University of Patras, Department of Mathematics, GR-26500 Rion, Greece}
\email{azachos@gmail.com}
\keywords{Fermat coordinate system, Fermat tree, Steiner tree, Intermediate tree, Platonic tree, $K-$space} \subjclass[2010]{Primary 05C05; Secondary 52B11}

\begin{abstract}  We call Platonic geodesic trees for a boundary vertex set $V$ and an added vertex set $Q$ in the $K-$Space ( $N-$dimensional Euclidean Space $\mathbb{R}^{N}$ with zero constant sectional curvature $K=0$ or the $n-$dimensional sphere $S_{K}^{N}$ with a positive constant sectional curvature equal to $K$ with radius $R=\frac{1}{\sqrt{K}}$ or the $n-$dimensional hyperbolic space $H_{K}^{N}$ with negative constant sectional curvature equal to $-K,$ for $K>0.$) a family of geodesic trees, whose each two incident edges meet at an angle $90^{o}$ and/or $120^{o},$ and/or $150^{o}$ and/or $180^{o}.$ Let $m$ be the number of connections (geodesic segments or geodesic rays) of a point in the $K-$Space. We prove that:

(A) If $m\ge 4,$ there exists a Platonic geodesic tree having $m$ edges (geodesic segments), which form $m$ platonic angles with measure $120^{o}$ and $(\frac{m(m-1)}{2}-m)$ platonic angles with measure $90^{o}.$ (Existence of a Fermat Coordinate system)

(B) If $m=2,$ there exists a Platonic geodesic tree with three vertices and 2 edges, which meet at $120^{\circ}$ or $150^{\circ}$ or $180^{\circ}$ (Minimal spanning geodesic tree).

(C) If $m=3,$ there exists a Platonic tree and the Steiner point is identical with the Fermat point, having three edges, which meet at $120^{\circ}$ (Hexagonal coordinate system).

(D) If $m$ is even and $m \ge 4,$  there exists a Platonic geodesic tree having $m$ edges , which form platonic angles with only two distinct measures $90^{o}$ and $180^{\circ}.$
(Existence of a Cartesian coordinate system).

(E) If $m\ge 4$, there does not exist a Platonic geodesic tree having $m$ edges , which form only platonic angles with measure $120^{o}$
(Non existence of a Fermat coordinate system).

(F) If $m\ge 3,$ there does not exist a Platonic geodesic tree having $m$ edges, which form platonic angles with only one measure $150^{o},$ or two distinct measures $120^{\circ}$ and $150^{\circ}$ or two distinct measures $90^{\circ}$ and $150^{\circ}$ (Non existence of a Fermat coordinate system).

A classification of platonic geodesic trees with boundary vertex set $V$ and added vertex set $Q$ in the $K-$Space is given by conditions (A), (B), (C), (D), (E), (F), which characterize each vertex in $V\cup Q.$

\end{abstract}\maketitle

\section{Introduction}

The $n-$dimensional $K$-Space is the $N-$dimensional Euclidean Space $\mathbb{R}^{N}$ with zero constant sectional curvature $K=0,$ the $n-$dimensional sphere $S_{K}^{N}$ with a positive constant sectional curvature equal to $K$ with radius $R=\frac{1}{\sqrt{K}},$ the $n-$dimensional hyperbolic space $H_{K}^{N}$ with negative constant sectional curvature equal to $-K,$ for $K>0.$

A Steiner minimal tree ($SMT_{K}(V)$) for $m$ points $A_{1},A_{2},\ldots, A_{m}$ (boundary vettices or terminals) in the $K-$Space is a geodesic tree whose vertex set $V$ contains these points having minimal length,i.e., the sum of the lengths of all its geodesic (shortest) segments (edges) are as small as possible. Added vertices other than the boundary vertices are called Steiner points. Steiner points are interior points of the convex hull of the $m$ distinct boundary vertices. A geodesic tree is called a Steiner tree ($ST_{K}(V)$) if it satisfies the following properties (\cite{Cieslik:05}, \cite[Theorem~2.1,p.~120,]{IvanTuzh:094}, \cite{NayaInnami:}:

(1) Any two edges meet at an angle at least $120^{o}$

(2) There are at most $m-2$ Steiner points.

(3) Each Steiner point has exactly three incident edges (degree three). The edges emanating from it to three neighboring vertices are unique geodesic segments.

(4) The geodesic tree has at most $2m-3$ edges. Equality holds if the boundary vertices of the vertex set $V$ have one connection (degree one).

Properties (1) and (3) imply that each Steiner point is incident to three geodesic segments meeting at $120^{o}$ (\cite{GilbertPollak:68}).
We note that the length of a geodesic segment connecting a boundary ideal vertex located at the boundary of the hyperbolic space $\partial H^{N}$ with a Steiner point equals with the length of a geodesic ray and is infinite. Referring to the other cases, the length of the geodesic segment connecting two points in the $K-$Space is finite.

An $ST_{K}(V)$ is called a full Steiner tree $FST_{K}(V)$ if has exactly $m-2$ Steiner points. The topology of $ST_{K}$ is a connection matric specifying which pairs of the vertex set $V=\{A_{1},A_{2},\ldots,A_{m}\}$ and the set of Steiner points $Q=\{S_{1},S_{2},\ldots, S_{k}\},$ for $0 \le k \le m-2$ have a connecting geodesic segment. The description of topologies of $ST_{0}(V)$ and the decomposition of a $ST_{0}(V)$ in $FST_{0} (V)$ trees in $\mathbb{R}^{2}$ is given in \cite{GilbertPollak:68} and generalized in metric spaces and manifolds in \cite{IvanTuzh:094}, \cite{IvanovTuzhilin:01b} and \cite{IvanovTuzhilinCieslik:03}.

A spanning tree on the vector set $V$ is a geodesic tree interconnecting all the boundary vertices of $V$ with geodesic segments. A geodesic segment connects two boundary vertices. The shortest spanning tree is called a minimal spanning tree $MST_{K}(V)$ on $V$ and $MST_{K}(V\cup Q)=SMT_{K}(V )$

Hwang and Weng (\cite{HwangWeng:86}, \cite[1.3,p.~154-158]{IvanTuzh:094} ) introduced a hexagonal coordinate system using three axes $(U,V,W)$ $120^{o}$ apart $120^{o},$ and gave an efficient method to construct a $ST_{0}(V)$ with a given topology.

A geodesic tree on the vertex set $V$ with only one added vertex (Fermat point) is called a Fermat tree $FT_{K}(V)$. A Fermat tree has $m$ geodesic segments interconnecting this point with the boundary vertices (\cite[1.1,pp.~1-11]{Cieslik:05}\cite{NodaSakaiMorimoto:91}).

An intermediate tree $IT_{K}(V)$ for $m$ points $A_{1},A_{2},\ldots, A_{m}$ (boundary vettices or terminals) in the $K-$Space is a geodesic tree whose vertex set $V$ contains these points having minimal length and the added vertices are allowed to have connections upto $m.$ The added vertices may be Fermat points and/or Steiner points. Some constructions of intermediate trees in $\mathbb{R}^{3}$ are given in \cite{Zachos:14} and \cite{Zachos:17}.

Plato introduced two Platonic (right) triangles with angles ($90^{o}-45^{o}-45^{o}$) and $90^{o}-30^{o}-60^{o}$ in his dialogue Timaeus (\cite{Plato:61} ), which compose the Platonic solids (regular polyhedra): a tetrahedron, cube, octahedron and icosahedron. By gluing six Platonic triangles $90^{o}-30^{o}-60^{o},$ we get an equilateral (regular) triangle. We call Platonic angle an angle with measure $30^{o}$ or $60^{o}$ or $90^{o}$ or $120^{o}$ or $150^{o}.$
By using the third, fourth, fifth and sixth multiple ($90^{o},$ $120^{o},$ $150^{\circ},$ $180^{\circ}$) of the platonic angle with measure $30^{o},$ we shall introduce the Platonic geodesic trees in the $K-$Space.

Platonic geodesic trees $PT_{K}(V),$ is a class of $IT_{K}(V),$ such that the platonic angles formed by two incident edges have measure $90^{o}$ and/or $120^{o},$ and/or $150^{\circ},$ and/or $180^{\circ}.$
Almost Platonic geodesic trees $APT_{K}(V),$ is a class of $IT_{K}(V),$ such that the angles formed by two incident edges have measure $90^{o}$ and/or $120^{o}$ and/or $180^{\circ}$ and there exists at least one angle with measure ($90^{\circ}+\epsilon)^{o},$ where $\epsilon \ne 30,60,90 $ is a positive real number.

Examples of Platonic geodesic trees are
(1) $FST_{K}(V),$ such that each two incident edges meet at an angle $120^{\circ}$

(2) An $IT_{K}(V),$ whose boundary vertices $2n$ are the vertices of an orthoplex having one added vertex (Fermat-point) with $2n$ connections. Each two incident edges meet at an angle $90^{o}$ or $180^{\circ}.$

(3) A $MST_{K}(V)$ with $m$ boundary vertices located on a geodesic ray in the $K-Space.$ Each two incident edges meet at $180^{\circ}.$

(4) A $MST_{0}(V)$ with 6 boundary vertices forming a regular hexagon in $\mathbb{R}^{2}.$ Each two incident edges meet at $120^{\circ}.$

 A $ST_{K}(V),$ which is not a $FST_{K}(V)$ is an $APT_{K}(V),$ if there are at least two incident edges, which meet at $(120+\epsilon)^{\circ}$
 Examples of an $APT_{K}(P)$ are two geodesic segments (or geodesic rays in $H^{2}$), which meet at $(120+\epsilon)^{\circ},$ or a $MST_{0}(V)$ with 10 boundary vertices forming a regular decagon in $\mathbb{R}^{2}.$ Each two incident edges meet at $144^{\circ}.$

Let $O$ be a point in the $n-$dimensional $K$-Space and  consider $m$ geodesic rays $l_{1}, l_{2},\ldots, l_{m}$ issuing from it, such that they do not all lie inside a single half-space. We will consider a point $A_{i},$ which may move along each geodesic ray $l_{i},$ for $i=1,2,\ldots, m.$ These vertices may form a closed geodesic simplex $A_{1}A_{2}\ldots A_{m}.$ A classical study for closed simplexes with vertices on prescribed rays is given in \cite[Chapter~9]{Alexandrov:05}.

The main purpose of this paper is to prove the following theorems:

Theorem~1 Let $l_{1},l_{2},\ldots,l_{m}$ geodesic rays issuing from $O$ in the $m-1$ dimensional $K-Space.$

(A) If $m\ge 4,$ there exists a Platonic geodesic tree $PT_{K}(V)$ having $m$ edges (geodesic segments), which form $m$ platonic angles with measure $120^{o}$ and $(\frac{m(m-1)}{2}-m)$ platonic angles with measure $90^{o}.$ (Existence of a Fermat Coordinate system)

(B) If $m=2,$ there exists a Platonic geodesic tree $PT_{K}(V)$ with three vertices and 2 edges, which meet at $120^{\circ}$ or $150^{\circ}$ or $180^{\circ}$ (Minimal spanning geodesic tree).

(C) If $m=3,$ there exists a Platonic tree $PT_{K}(V)$ and the Steiner point is identical with the Fermat point, having three edges, which meet at $120^{\circ}$ (Hexagonal coordinate system).

(D) If $m$ is even and $m \ge 4,$  there exists a Platonic geodesic tree $PT_{K}(V)$ having $m$ edges , which form platonic angles with only two distinct measures $90^{o}$ and $180^{\circ}.$
(Existence of a Cartesian coordinate system).

(E) If $m\ge 4$, there does not exist a Platonic geodesic tree $PT_{K}(V)$ having $m$ edges , which form only platonic angles with measure $120^{o}$
(Non existence of a Fermat coordinate system).

(F) If $m\ge 3,$ there does not exist a Platonic geodesic tree $PT_{K}(V)$ having $m$ edges, which form platonic angles with only one measure $150^{o},$ or two distinct measures $120^{\circ}$ and $150^{\circ}$ or two distinct measures $90^{\circ}$ and $150^{\circ}$ (Non existence of a Fermat coordinate system).


Theorem~2

If an $IT_{K}(V)$ has $m$ boundary vertices with upto $m-2$ added vertices that lie on their convex geodesic hull and the geodesic rays formed by the edges of this geodesic tree give a Fermat Coordinate System and/or Cartesian Coordinate System and/or Hexagonal coordinate system, then $IT_{K}(V)=PT_{K}(V).$




\section{Platonic geodesic trees in the $n-$ dimensional K space }
In this section, we give a characterization of Platonic geodesic trees $PT_{K}(V)$ with $m$ vertices $V=\{ A_{1},  A_{2},\ldots, A_{m} \}$ in the $K-$Space having one added vertex $Q=\{ O \}$ with  $m$ connections geodesic segments (or rays) $A_{i}O,$ such that each two incident edges meet an angle $\angle A_{i}OA_{j},$ which takes values in $\{90^{\circ}, 120^{\circ}, 150^{\circ}, 180^{\circ} \}$ (Theorem~\ref{theorem1}). By applying theorem~\ref{theorem1} for each vertex in $V\cup Q,$ where $Q$ is the set of added vertices (Fermat points or Steiner points), we introduce a classification of Platonic geodesic trees $PT_{K}(V)$ in the $K-$Space (Theorem~\ref{theorem2}).

Let $O$ be a point in the $n-$dimensional $K$-Space and  consider $m$ geodesic rays $l_{1}, l_{2},\ldots, l_{m}$ issuing from it, such that they do not all lie inside a single half-space. We will consider a point $A_{i},$ which may move along each geodesic ray $l_{i},$ for $i=1,2,\ldots, m.$ These vertices may form an $m-$simplex $\Delta^{m}(V)=A_{1}A_{2}\ldots A_{m}.$

\begin{theorem}\label{theorem1}
The following conditions characterize a $PT_{K}(V)$ with one added vertex $O$ having $m$ edges issuing from $O,$ (except $m=2,$ where $O$ coincides with one vertex frome the vertex set of three vertices):

A) If $m\ge 4,$ there exists a Platonic geodesic tree $PT_{K}(V)$ having $m$ edges (geodesic segments), which form $m$ platonic angles with measure $120^{o}$ and $(\frac{m(m-1)}{2}-m)$ platonic angles with measure $90^{o}.$ (Existence of a Fermat Coordinate system)

(B) If $m=2,$ there exists a Platonic geodesic tree $PT_{K}(V)$ with three vertices and 2 edges, which meet at $120^{\circ}$ or $150^{\circ}$ or $180^{\circ}$ (Minimal spanning geodesic tree).

(C) If $m=3,$ there exists a Platonic tree $PT_{K}(V)$ and the Steiner point is identical with the Fermat point, having three edges, which meet at $120^{\circ}$ (Hexagonal coordinate system).

(D) If $m$ is even and $m \ge 4,$  there exists a Platonic geodesic tree $PT_{K}(V)$ having $m$ edges , which form platonic angles with only two distinct measures $90^{o}$ and $180^{\circ}.$
(Existence of a Cartesian coordinate system).

(E) If $m\ge 4$, there does not exist a Platonic geodesic tree $PT_{K}(V)$ having $m$ edges , which form only platonic angles with measure $120^{o}$
(Non existence of a Fermat coordinate system).

(F) If $m\ge 3,$ there does not exist a Platonic geodesic tree $PT_{K}(V)$ having $m$ edges, which form platonic angles with only one measure $150^{o},$ or two distinct measures $120^{\circ}$ and $150^{\circ}$ or two distinct measures $90^{\circ}$ and $150^{\circ}$ (Non existence of a Fermat coordinate system).

\end{theorem}

\begin{proof}

(A)
Let $e_{i}=\frac{\vec{OA_{i}}}{| OA_{i} |}$ be a unit direction tangent vector of the ray $l_{i}$ at $O,$ for $i=1,2,\ldots, m.$ $O$ is the Fermat point having $m$ connections with the geodesic segments $OA_{i}.$ Hence, we have the balancing condition of $e_{i}$ at $O:$

\begin{equation}\label{eqone}
\sum_{i=1}^{m}e_{i}=\vec{0}.
\end{equation}
By squaring both parts of (\ref{eqone}), we need to set $m$ angles with measure $120^{\circ}$ and the rest $\frac{m(m-1)}{2}-m$ angles with measure $90^{\circ}.$

Let $\{ \cos \angle A_{i}O A_{j} \}$ be the cosine matrix of the angles $\angle A_{i} O A_{j},$ for $i,j=1,2,\ldots, m.$
We set $\cos \angle A_{i}OA_{i}=1,$  for $i=1,2,\ldots, m,$ \[\angle A_{j} O A_{j+1}=\angle A_{j+1} O A_{j}=120^{\circ},\] for $j=1,2,3,\ldots, m-1,$ which yields \[\cos \angle A_{j} O A_{j+1}=\cos \angle A_{j+1} O A_{j}=-\frac{1}{2}.\]

We set \[\angle A_{1}O A_{m}=\angle A_{m}OA_{1}=120^{\circ},\] which yields \[\cos \angle A_{1} O A_{m}=\cos \angle A_{m} O A_{1}= -\frac{1}{2}.\]

We set the rest $\frac{m(m-1)}{2} - m$ angles \[\angle A_{i} O A_{j} =90^{\circ},\] which yields \[\cos \angle A_{i}O A_{j}=0,\] for $i<j.$
Hence, we get \[\angle A_{j} O A_{i} =90^{\circ},\] which yields \[\cos \angle A_{j}O A_{i}=0.\]
By following this construction, we get exactly two elements of the cosine matrix with value $-\frac{1}{2}$ one with value $1$ and the rest with value $0.$
in each row and each column.

These values satisfy the following combinatorial system of equations of unit tangent vectors $e_{i}$ at $O,$ derived by (\ref{eqone}):

\[-e_{i_{1}}=\sum_{i=1, i\ne i_{1} }^{m} e_{i},\]
\[-e_{i_{1}}-e_{i_{2}}=\sum_{i=1, i\ne \{i_{1}, i_{2}\} }^{m} e_{i},\]
\[-e_{i_{1}}-e_{i_{2}}-e_{i_{3}}=\sum_{i=1, i\ne \{i_{1}, i_{2},i_{3}\} }^{m} e_{i},\]
$\cdots$
\[-e_{i_{1}}-e_{i_{2}}-\ldots-e_{i_{k}}=\sum_{i=1, i\ne \{i_{1}, i_{2},i_{3},\ldots, i_{k}\} }^{m} e_{i},\]
for $i_{1},i_{2},\ldots, i_{k} \in \{1,2,\ldots, m\}$ and $k<\frac{m}{2}.$

(B)  $m=2$ Two geodesic segments $A_{1}A_{2}$ and $A_{2}A_{3}$ (or rays in $H_{K}^{2}$), which meet at angle $\angle A_{1}A_{2}A_{3}=120^{\circ}$ or $150^{\circ}$ or $180^{\circ}$ gives a minimum spanning tree $MST_{K}( \{ A_{1}, A_{2}, A_{3} \} )  )$ and $O=A_{2}$ is the Fermat point (degenerate Fermat tree or Steiner tree), which coincides with a $PT_{K}(\{ A_{1}, A_{2}, A_{3} \}).$

(C) For $m=3,$ consider a vertex set of three vertices, which form a geodesic triangle with angles $120^{\circ}-\epsilon,$ where $\epsilon$ is a positive real number. Hence, we get $e_{1}+e_{2}+e_{3}=\vec{0},$ such that: $\angle A_{1}A_{0}A_{2}=\angle A_{2}A_{0}A_{3}=\angle A_{1}A_{0}A_{3}=120^{\circ}$ and $PT_{K}(\{ A_{1}, A_{2}, A_{3} \})=MST_{K}(\{ A_{1}, A_{2}, A_{3} \})$ with one added vertex $O$ (Fermat point or Steiner point).

(D) For $m$ even, we get an $\frac{m}{2}$ orthoplex having $m$ vertices and $\frac{m}{2}$ diagonals, which meet at a right angle. Hence, the edges of the geodesic tree form $m (\frac{m}{2}-1)$ right angles and $\frac{m}{2}$ straight angles, which yields a $PT_{K}(V)$.

(E) For $m\ge 4,$ we assume that $\frac{m(m-1)}{2}$ $\angle A_{i}O A_{j}$  have measure $120^{\circ}$ for $i,j=1,2, \ldots, m$ and $ i<j.$ By squaring both parts of the balancing condition of $e_{i}$ in(\ref{eqone}), we obtain:
\[ m - \frac{m (m-1) }{2} = 0 ,\]

 which gives $m=3 <4 .$
Therefore, there does not exist a $PT_{K}(V)$ with $\frac{m(m-1)}{2}$ angles with measure $120^{\circ}$ for $m\ge 4.$

(F) For $m\ge 3,$ we assume that there is at least one angle $\angle A_{i}OA_{j}$ with measure $150^{\circ}.$ By squaring both parts of the balancing condition of $e_{i}$ in(\ref{eqone}), we obtain that the irrational number $m +a + b\sqrt{3}=0,$ for $a,b$ real numbers. Hence, there does not exist a $PT_{K}(V),$ such that each two incident edges meet at an angle having only one measure $150^{o},$ or two distinct measures $120^{\circ}$ and $150^{\circ}$ or two distinct measures $90^{\circ}$ and $150^{\circ}.$

\end{proof}

\begin{theorem}\label{theorem2}
If an $IT_{K}(V)$ has $m$ boundary vertices with upto $m-2$ added vertices that lie on their convex geodesic hull and the geodesic rays formed by the edges of this geodesic tree give a Fermat Coordinate System and/or Cartesian Coordinate System and/or Hexagonal coordinate system, then $IT_{K}(V)=PT_{K}(V).$
\end{theorem}

\begin{proof}
Let an intermediate geodesic tree $IT_{K}(V)$ with $m$ boundary vertices and an added set of vertices $Q,$ which contains upto $m-2$ vertices in the $K-$space. Each point $S_{i}$ in $Q$ may have $k_{i},$ for $i=1,2,\ldots, m-2$ and refers to the conditions of (A), (B), (C), (D) of Theorem~\ref{theorem1}. There might be some vertices of the original vertex set $V$ that fullfill case (B) of Theorem~1. Hence, we get: $IT_{K}(V)=PT_{K}(V).$
\end{proof}

\begin{proposition}\label{propthree}
If $m$ is even, such that: $\angle A_{j} O A_{j+2}=\angle A_{1} O A_{m}=\angle A_{2}OA_{m-1}=120^{\circ},$ for $j=1,2,3,\ldots, m-2,$ and the rest $\frac{m (m-1)}{2}-m$ angles $A_{i}OA_{j}=90^{\circ},$ there exists a $PT_{K}(V)$ with $m$ edges, which form $m$ platonic angles with measure $120^{o}$ and $(\frac{m(m-1)}{2}-m)$ platonic angles with measure $90^{o}.$ (Alternative construction of a Fermat coordinate system).
\end{proposition}

\begin{proof}
Let $\{ \cos \angle A_{i}O A_{j} \}$ be the cosine matrix of the angles $\angle A_{i} O A_{j},$ for $i,j=1,2,\ldots, m.$
We set $\cos \angle A_{i}OA_{i}=1,$  for $i=1,2,\ldots, m,$ \[\angle A_{j} O A_{j+2}=\angle A_{j+2} O A_{j}=120^{\circ},\] for $j=1,2,3,\ldots, m-2,$ which yields \[\cos \angle A_{j} O A_{j+2}=\cos \angle A_{j+2} O A_{j}=-\frac{1}{2}.\]

We set \[\angle A_{1}O A_{m}=\angle A_{m}OA_{1}=120^{\circ}\] and \[ \angle A_{2}O A_{m-1}=\angle A_{m-1}OA_{2}=120^{\circ},\] which yields \[\cos \angle A_{1} O A_{m}=\cos \angle A_{m} O A_{1}=\cos \angle A_{2} O A_{m-1}=\cos \angle A_{m-1} O A_{2}=-\frac{1}{2}.\]

We set the rest $\frac{m(m-1)}{2} - m$ angles \[\angle A_{i} O A_{j} =90^{\circ},\] which yields \[\cos \angle A_{i}O A_{j}=0,\] for $i<j.$
Hence, we get \[\angle A_{j} O A_{i} =90^{\circ},\] which yields \[\cos \angle A_{j}O A_{i}=0.\]
By following this construction, we get exactly two elements of the cosine matrix with value $-\frac{1}{2}$ one with value $1$ and the rest with value $0.$
in each row and each column.

These values satisfy the following combinatorial system of equations of unit tangent vectors $e_{i}$ at $O,$ derived by (\ref{eqone}) for $m$ even:

\[-e_{i_{1}}=\sum_{i=1, i\ne i_{1} }^{m} e_{i},\]
\[-e_{i_{1}}-e_{i_{2}}=\sum_{i=1, i\ne \{i_{1}, i_{2}\} }^{m} e_{i},\]
\[-e_{i_{1}}-e_{i_{2}}-e_{i_{3}}=\sum_{i=1, i\ne \{i_{1}, i_{2},i_{3}\} }^{m} e_{i},\]
$\cdots$
\[-e_{i_{1}}-e_{i_{2}}-\ldots-e_{i_{\frac{m}{2}}}=\sum_{i=1, i\ne \{i_{1}, i_{2},i_{3},\ldots, i_{\frac{m}{2}}\} }^{m} e_{i},\]
for $i_{1},i_{2},\ldots, i_{\frac{m}{2}} \in \{1,2,\ldots, m\}.$

\end{proof}




\begin{corollary}\label{PTtetrahedron}
If $m=4,$ such that: $\angle A_{1}A_{0}A_{2}=\angle A_{3}A_{0}A_{4}=90^{\circ}$ and
$\angle A_{2}A_{0}A_{3}=\angle A_{1}A_{0}A_{4}=\angle A_{1}A_{0}A_{3}=\angle A_{2}A_{0}A_{4}=120^{\circ}$

or $\angle A_{1}A_{0}A_{3}=\angle A_{2}A_{0}A_{4}=90^{\circ}$ and
$\angle A_{1}A_{0}A_{2}=\angle A_{2}A_{0}A_{3}=\angle A_{3}A_{0}A_{4}=\angle A_{1}A_{0}A_{4}=120^{\circ}$

there exists a $PT_{K}(\{A_{1},A_{2},A_{3,}A_{4}\})$ with four edges, which form 4 platonic angles with measure $120^{\circ}$
and 2 platonic angles with measure $90^{\circ}.$
\end{corollary}

\begin{proof}
We obtain two different constructions, by applying Theorem~\ref{theorem1} Case (A) and Proposition~\ref{propthree}, for $m=4.$
\end{proof}

\begin{proposition}\label{Propfour}
A $PT_{0}(\{A_{1},A_{2},A_{3,}A_{4}\})$ with one added vertex $O$ cannot be prescribed to a Platonic solid (regular tetrahedron) with vertices $\{A_{1},A_{2},A_{3},A_{4}\}$ in $\mathbb{R}^{3}.$
\end{proposition}

\begin{proof}
Let $A_{i}O=a_{i},$ $A_{i}A_{j}=1$ and $\angle A_{1}A_{0}A_{3}=\angle A_{2}A_{0}A_{4}=90^{\circ},$
$\angle A_{1}A_{0}A_{2}=\angle A_{2}A_{0}A_{3}=\angle A_{3}A_{0}A_{4}=\angle A_{1}A_{0}A_{4}=120^{\circ}.$
By using the cosine law in $\triangle A_{i}OA_{j},$ we obtain the following system of equations with respect to $a_{1},a_{2},a_{3},a_{4}:$

\begin{equation}\label{eqt12}
a_{1}^2+a_{2}^2=1,
\end{equation}
\begin{equation}\label{eqt13}
a_{1}^2+a_{3}^2 +a_{1}a_{3}=1,
\end{equation}
\begin{equation}\label{eqt14}
a_{1}^2+a_{4}^2 +a_{1}a_{4}=1,
\end{equation}
\begin{equation}\label{eqt24}
a_{2}^2+a_{4}^2 +a_{2}a_{4}=1,
\end{equation}

The system of four equations (\ref{eqt12}), (\ref{eqt13}), (\ref{eqt14}), (\ref{eqt24}  does not have a positive real solution with respect to $(a_{1},a_{2},a_{3},a_{4}).$

\end{proof}

\begin{proposition}\label{Propfive}
An $APT_{0}(\{A_{1},A_{2},A_{3,}A_{4}\})$ with one added vertex $O$ and each two edges form a Platonic angle with measure $\arccos{-\frac{1}{3}}$ can be prescribed to a Platonic solid (regular tetrahedron) with vertices $\{A_{1},A_{2},A_{3},A_{4}\}$ in $\mathbb{R}^{3}.$
\end{proposition}

\begin{proof}
By applying the cosine law in $\triangle A_{i}OA_{j}$ and taking into account that $A_{i}A_{j}=1,$ for $i=j=1,2,3,4,$ we obtain that
$A_{i}O=\frac{\sqrt{6}}{4}.$
\end{proof}

\begin{proposition}\label{Prop6}
A $PT_{0}(\{A_{1},A_{2},A_{3,}A_{4}\})$ with two added vertices $O$ and $O^{\prime}$ can be prescribed to a Platonic solid (regular tetrahedron) with vertices $\{A_{1},A_{2},A_{3},A_{4}\}$ in $\mathbb{R}^{3}.$
\end{proposition}

\begin{proof}
The added vertices $O$ and $O^{\prime}$ are Steiner points having degree three, whose incident edges meet at $120^{\circ}.$
Hence, $MST_{0}(\{A_{1},A_{2},A_{3,}A_{4}\} )=PT_{0}(\{A_{1},A_{2},A_{3,}A_{4}\})$ can be prescribed to a regular (Platonic) tetrahedron $(\{A_{1},A_{2},A_{3,}A_{4}\}.$
\end{proof}





\end{document}